\documentclass[10pt]{amsart}
\usepackage{amsmath}
\usepackage{amssymb}
\usepackage{amsthm}
\usepackage{amscd}

\newtheorem{theorem}{Theorem}
\newtheorem{lemma}[theorem]{Lemma}
\newtheorem{conjecture}[theorem]{Conjecture}
\newtheorem{proposition}[theorem]{Proposition}

\numberwithin{theorem}{section}

\theoremstyle{definition}

\newtheorem{example}[theorem]{Example}
\newtheorem{remark}[theorem]{Remark}

\numberwithin{equation}{section}

\newcommand\simarrow{\stackrel{\sim}{\rightarrow}}

\newcommand\xb{\overline{x}}
\newcommand\yb{\overline{y}}
\newcommand\Z{\mathbb{Z}}

\begin{document}

\begin{center}
{\Large{\bf Generalized energies and 
integrable $D^{(1)}_n$ cellular automaton}}
\vspace{3mm}\\
{\large Atsuo Kuniba, Reiho Sakamoto and Yasuhiko Yamada}
\vspace{3mm}
\end{center}

\begin{center}
\begin{small}
{\it Dedicated to  Professor Tetsuji Miwa
on his 60th birthday}
\end{small}
\end{center}

\author{Atsuo Kuniba}

\address{Institute of Physics, University of Tokyo, Komaba\\
Tokyo 153-8902, Japan\\
}

\author{Reiho Sakamoto}

\address{Department of Physics, Tokyo University of Science\\
Tokyo 162-8601, Japan}

\author{Yasuhiko Yamada}
\address{Department of Mathematics, Faculty of Science,
Kobe University\\
 Hyogo 657-8501, Japan}
\begin{quotation}
{\small
A}{\tiny BSTRACT:}
{\small
We introduce generalized energies for a class of 
$U_q(D^{(1)}_n)$ crystals 
by using the piecewise linear functions 
that are building blocks of the combinatorial $R$.
They include the conventional energy in 
the theory of affine crystals as a special case.
It is shown that the generalized energies count 
the particles and anti-particles 
in a quadrant of the two dimensional lattice
generated by time evolutions 
of an integrable $D^{(1)}_n$ cellular automaton.
Explicit formulas are conjectured for some of them  
in the form of ultradiscrete tau functions.
}
\end{quotation}

\keywords{crystal base, integrable cellular automaton, generalized energy,
combinatorial Bethe ansatz, inverse scattering method, 
ultradiscrete tau function}

\section{Introduction}\label{sec:intro}

Let $B_l$ be the crystal of the 
$l$-fold symmetric tensor representation of 
the quantum affine algebra 
$U_q(D^{(1)}_n)$ \cite{KMN2,KKM}.
The combinatorial $R$ : 
$x\otimes y\mapsto y' \otimes x'$
is the isomorphism of crystals
$B_l \otimes B_m \stackrel{\sim}{\rightarrow} 
B_m \otimes B_l$ 
corresponding to the quantum $R$ at $q=0$ \cite{KMN1}.
In Ref. \cite{KOTY1}, an explicit formula of the combinatorial $R$ 
was obtained in terms of 
several piecewise linear functions $g_i(x\otimes y)\in \Z$
on $B_l \otimes B_m$.
See Theorem \ref{th:koty}.
Among them is the {\em local energy}, 
which plays an essential role 
in the theory of affine crystals \cite{KMN1}.
The family of piecewise linear functions $\{g_i\}$, which we call 
{\em generalized local energies} in this paper,
are ultradiscretization of the 
subtraction-free rational functions that have 
emerged as building blocks of the 
tropical $R$ \cite{KOTY1,KOTY2} of the 
geometric crystal \cite{BK}.
They may be viewed as local energies 
in a principal picture rather than in the conventional  
homogeneous picture.

{}From the local energy, one can form the 
integer-valued function called 
{\em energy} on the tensor product 
${\mathcal P}=B_{l_1}\otimes  \cdots \otimes B_{l_L}$.
Its generating function is the one dimensional configuration sum that
originates in the corner transfer matrix method \cite{ABF,Ba}.

In this paper we introduce {\em generalized energies} 
${\mathcal E}_{g_i}: {\mathcal P} \rightarrow \Z_{\ge 0}$
corresponding to $g_i$'s, and study them from the viewpoint of 
the integrable cellular automaton of type $D^{(1)}_n$ \cite{HKT1,HKT2}.
The latter is an integrable $U_q(D^{(1)}_n)$ vertex model 
at $q=0$.
It is a dynamical system on ${\mathcal P}$
equipped with commuting time evolutions $\{T_l\}_{l \ge 1}$.
Elements of ${\mathcal P}$ are naturally regarded as 
arrays of particles and anti-particles, and 
$T_l$ induces their factorized scattering involving 
pair creation and annihilation.
See Examples \ref{ex:hatten1} and \ref{ex:hatten2}.

Our main result is Theorem \ref{th:main}, which states that 
${\mathcal E}_{g_i}(p) = \rho_{g_i}(p)$ for any 
$p=p_1\otimes \cdots \otimes p_L \in {\mathcal P}$.
Here $\rho_{g_i}(p)$ is a  
{\em counting function} giving 
the number of certain particles and anti-particles specified by $g_i$  
in the region (\ref{sw}) 
under the time evolutions $p, T_\infty(p), T^2_\infty(p),\ldots$.
As such, the counting functions are non-local variables attached to a
quadrant of the 2 dimensional lattice.
However, it will also be shown in Theorem \ref{pr:rho} that the combined data 
$\{\rho_{g_i}(p_1\otimes \cdots \otimes p_j)\mid j=k-1,k\}$   
in turn reproduces the local variable $p_k \in B_{l_k}$ completely
in agreement with the spirit of the corner transfer matrix method.
Therefore the joint spectrum 
$\{{\mathcal E}_{g_i}(p_1\otimes \cdots \otimes p_k)\}$ 
of the generalized energies with $1 \le k \le L$ is 
equivalent to $p=p_1\otimes \cdots \otimes p_L\in {\mathcal P}$ itself.
This extends a similar result on type $A^{(1)}_n$
(Proposition 4.6 in Ref. \cite{KSY})
which is related to the katabolism \cite{Sh}. 
A supplementary result (Proposition \ref{pr:main2}) is parallel with 
Theorem \ref{th:main} and treats generalized local energies with
opposite chirality (cf. Remark \ref{re:chiral}). 

The layout of the paper is as follows.
In Section \ref{sec:ge}, generalized (local) energies are extracted 
from the piecewise linear formula of the combinatorial $R$ \cite{KOTY1}.
In Section \ref{sec:sca}, the integrable $D^{(1)}_n$ 
cellular automaton \cite{HKT1,HKT2} is recalled     
and the counting functions are defined.
In Section \ref{sec:main}, 
the main Theorem \ref{th:main} of the paper is stated and proved.
In Section \ref{sec:re},  aspects related to 
combinatorial Bethe ansatz are discussed.
In Section \ref{ss:ist} we give 
the inverse scattering formalism 
of the $D^{(1)}_n$ cellular automaton like Ref. \cite{KOSTY}.
In Section \ref{ss:tau},
we conjecture piecewise linear formulas for 
some generalized energies in terms of {\em ultradiscrete tau functions}.
This is also motivated by the $A^{(1)}_n$ case \cite{KSY},
where analogous results have led to a piecewise linear formula for the 
Kerov-Kirillov-Reshetikhin map \cite{KR}.
Although the conjecture is yet to cover the full family of 
generalized energies, the last one (\ref{cj5}) is already rather intriguing.
We expect that the extention and the solution of Conjecture \ref{cj:tau}
will uncover an interplay among combinatorial Bethe ansatz, 
ultradiscretization of the DKP hierarchy \cite{JM} and the bilinearlization of 
the tropical $R$ \cite{KOTY2}. 

\section{Generalized energies for $D^{(1)}_n$ crystal}\label{sec:ge}

\subsection{Crystals and combinatorial $R$}
Let us recall the basic facts on crystal and combinatorial $R$ briefly.
For a more information, see Refs. \cite{KMN1, KMN2, KKM} and \cite{O}.
For a positive integer $l$, let 
\begin{equation}\label{Bl}
B_l = \{ \zeta=(\zeta_1,\ldots,\zeta_n,\overline{\zeta}_n,\ldots,
\overline{\zeta}_1) \in \Z_{\ge 0}^{2n} \mid 
\sum_{i=1}^n(\zeta_i + \overline{\zeta}_i) = l, \; \;
\zeta_n\overline{\zeta}_n = 0 \}
\end{equation}
be the crystal of the 
$l$-fold symmetric tensor representation of 
$U_q(D^{(1)}_n)$ \cite{KKM}.
We assume $n \ge 3$.
As for the functions $\varepsilon_i, \varphi_i$, 
the tensor product rule and the action of Kashiwara operators 
$\tilde{e}_i$ and $\tilde{f}_i\, (0 \le i \le n)$, see Ref. \cite{O}.

The {\em affinization} of the crystal $B_l$ is defined by
${\rm Aff}(B_l)=\{b[d]\mid d\in\Z,b\in B_l\}$ with the crystal
structure $\tilde{e}_i(b[d])=(\tilde{e}_ib)[d+\delta_{i0}]$ and 
$\tilde{f}_i(b[d])=
(\tilde{f}_ib)[d-\delta_{i0}]$. 
We call $b$ and $d$ the 
classical and the affine part of $b[d]$, respectively.
There exists the unique bijection (crystal isomorphism)
$B_l\otimes  B_m\stackrel{\sim}{\rightarrow}
B_m\otimes B_l$ that commutes with all Kashiwara operators.
It is lifted up to a map 
${\rm Aff}(B_l)\otimes {\rm Aff}(B_m)
\stackrel{\sim}{\rightarrow}{\rm Aff}(B_m)\otimes{\rm Aff}(B_l)$
called the {\em combinatorial $R$}, which has the following form:
\begin{eqnarray*}
R\;:\;{\rm Aff}(B_l)\otimes{\rm Aff}(B_m)
&\longrightarrow&{\rm Aff}(B_m)\otimes{\rm Aff}(B_l)\\
b[d]\otimes b'[d']\;\;&\longmapsto&\;\;
\tilde{b}'[d'+H(b\otimes b')]
\otimes \tilde{b}[d-H(b\otimes b')],
\end{eqnarray*}
where $b\otimes b'\mapsto\tilde{b'}\otimes\tilde{b}$ 
under the isomorphism $B_l\otimes B_m
\stackrel{\sim}{\rightarrow}B_m\otimes B_l$. \footnote{
This classical part of the combinatorial $R$ will also be referred as 
combinatorial $R$ and denoted by 
$R(b\otimes b') = \tilde{b'}\otimes\tilde{b}$.}
The quantity $H(b\otimes b')$ is called the 
{\em local energy} and determined up to a global additive constant by
\[
H(\tilde{e}_i(b\otimes b'))=\left\{
\begin{array}{ll}
H(b\otimes b')+1&\mbox{ if }i=0,\ \varphi_0(b)\geq \varepsilon_0(b'), 
\varphi_0(\tilde{b}')\geq\varepsilon_0(\tilde{b}),\\
H(b\otimes b')-1&\mbox{ if }i=0,\ \varphi_0(b)<\varepsilon_0(b'),
\varphi_0(\tilde{b}')<\varepsilon_0(\tilde{b}),\\
H(b\otimes b')&\mbox{ otherwise}.
\end{array}\right.
\]
The Yang-Baxter equation 
\begin{equation}\label{ybe}
(R\otimes1)(1\otimes R)(R\otimes1)
=(1\otimes R)(R\otimes1)(1\otimes R)
\end{equation}
is satisfied on 
${\rm Aff}(B_l)\otimes{\rm Aff}(B_{m})
\otimes{\rm Aff}(B_{k})$.

\subsection{Generalized local energies}\label{ss:glef}
Let us give an explicit piecewise linear formula 
of the combinatorial $R$ that originates in the 
tropical $R$ for geometric crystals of type $D^{(1)}_n$ \cite{KOTY1}.
First we make a slight variable change.
The set $B_l$ (\ref{Bl}) is in one to one correspondence with 
another set
\begin{eqnarray}
B'_l &=& \{ x=(x_1,\ldots,x_n,\overline{x}_{n-1},\ldots,
\overline{x}_1) \in \Z^{2n-1} \mid 
x_i,\overline{x}_i\ge 0\text{ for }1\le i\le n-1,\nonumber\\
&&\qquad\qquad \quad x_n\ge-\min(x_{n-1},\overline{x}_{n-1}),
\sum_{i=1}^{n-1}(x_i + \overline{x}_i) + x_n = l\}
\label{Bp}
\end{eqnarray}
by the relations 
\begin{eqnarray}
x_i &=& \zeta_i, \;\;  \overline{x}_i 
= \overline{\zeta}_i \;\;  (1 \le i \le n-2),\label{xz4}\\
x_{n-1}&=&\zeta_{n-1}+\overline{\zeta}_{n}, \;
x_{n}=\zeta_{n}-\overline{\zeta}_{n}, \; 
\overline{x}_{n-1}=\overline{\zeta}_{n-1}+\overline{\zeta}_{n},
\label{xz5}\\ 
\zeta_n &=& \max(0,x_n),\; \overline{\zeta}_n=\max(0,-x_n),
\label{xz6}\\
\zeta_{n-1} &=& x_{n-1} + \min(0,x_n),\; 
\overline{\zeta}_{n-1} = \overline{x}_{n-1} + \min(0,x_n).
\label{xz7}
\end{eqnarray}
Note that $x_n$ can be negative.
We naturally use the notations like 
$x[d] \in {\rm Aff}(B'_l)$ and 
$R(x[d]\otimes x'[d'])
= \tilde{x}'[d'+H(x\otimes x')]\otimes 
\tilde{x}[d-H(x\otimes x')]$, etc.
Set
\begin{equation}
\ell(\zeta) = 
\sum_{i=1}^n(\zeta_i + \overline{\zeta}_i) \;\;(\zeta \in B_l),
\quad
\ell(x) = \sum_{i=1}^{n-1}(x_i + \overline{x}_i) + x_n 
\;\;(x \in B'_l),
\end{equation}
so that $\ell(\zeta) = \ell(x)=l$ for 
$\zeta \in B_l$ and $x \in B'_l$.

Let $x=(x_1,\ldots,\xb_1)\in B'_l$ and
$y=(y_1,\ldots,\yb_1)\in B'_m$.
On the pair $(x,y)$ we introduce mutually commuting 
involutions $\sigma_1, \sigma_n$ and $\ast$ by
\begin{eqnarray}
(x,y)^{\sigma_1} &=& (x^{\sigma_1},y^{\sigma_1}), \quad
(x,y)^{\sigma_n} = (x^{\sigma_n},y^{\sigma_n}), \quad
(x,y)^{*} = (y^{*},x^{*}),\label{sss}\\
\sigma_1 &: & x_1 \longleftrightarrow \xb_1, \nonumber\\
\sigma_n &: &x_{n-1} \rightarrow x_{n-1}+x_n, \quad
             \xb_{n-1} \rightarrow \xb_{n-1}+x_n, \quad
             x_n \rightarrow -x_n,\nonumber\\
\ast &: & x_i \longleftrightarrow \xb_i\quad (1 \leq i \leq n-1).
\label{sta}
\end{eqnarray}
The coordinates not included in the above rules are left unchanged.
These involutions are 
naturally defined on $(\xi,\zeta) \in B_l\times B_m$ as well 
by the correspondence (\ref{xz4})-(\ref{xz7}).
For instance, one has 
$(\xi, \zeta)^\ast = (\zeta^\ast, \xi^\ast)$ with 
$\zeta^\ast=(\overline{\zeta}_1,\ldots, \overline{\zeta}_{n-1},
\zeta_n, \overline{\zeta}_n, \zeta_{n-1},\ldots, \zeta_1)$ for 
$\zeta=(\zeta_1,\ldots, \zeta_n, 
\overline{\zeta}_n,\ldots, \overline{\zeta}_1)$.

For any function $g=g(x,y)$, we write
$g^{\sigma_1}=g^{\sigma_1}(x,y) 
= g(x^{\sigma_1}, y^{\sigma_1})$, etc.
Introduce the piecewise linear functions 
$V_i=V_i(x,y)$ and  
$W_i=W_i(x,y)$ for $0 \le i \le n-1$ as follows.
\begin{eqnarray}
&&V_i =  \max
\left( \{ \theta_{i,j},
\theta'_{i,j} | 1 \leq j \leq n-2 \} \cup
\{ \eta_{i,j}, \eta'_{i,j} | 1 \leq j \leq n \} \right),\label{Vi}
\\
&&W_0 = 2V_0,\;\;W_1 = V_0 + V_0^{\sigma_1},\;\; W_{n-1}=V_{n-1}+V_{n-1}^*, 
\label{W1}\\
&&W_i=\max \left( V_i+V_{i-1}^*-y_i, V_{i-1}+V_i^*-\xb_i \right)
+ \min (x_i, \yb_i ),\label{Wi}
\end{eqnarray}
where $2 \le i \le n-2$ in the last line.
The functions $\theta_{i,j}=\theta_{i,j}(x,y), 
\theta'_{i,j}=\theta'_{i,j}(x,y),
\eta_{i,j}=\eta_{i,j}(x,y), \eta'_{i,j}=\eta'_{i,j}(x,y)$
are defined by
\begin{eqnarray*}
\theta_{i,j}(x,y) &=&
\begin{cases}
\ell(x) +{\displaystyle
\sum_{k=j+1}^{i} (\yb_k - \xb_k)} & \mbox{for} \;\;
1 \leq j \leq i, \\
\ell(y) +{\displaystyle
\sum_{k=i+1}^{j} (\xb_k - \yb_k)} & \mbox{for} \;\;i+1 \leq j \leq n-2,
\end{cases}
\\
\theta'_{i,j}(x,y) &=& \ell(x) + \sum_{k=1}^{i} (\yb_k - \xb_k)+
\sum_{k=1}^{j} (y_k - x_k)\;\;
\mbox{for} \; 1\le j \le n-2,
\end{eqnarray*}
\begin{eqnarray*}
\eta_{i,j}(x,y) &=&
\begin{cases}
\ell(x) +{\displaystyle  \sum_{k=j+1}^{i} (\yb_k -\xb_k)  +
\yb_{j}- x_{j}} &\mbox{for} \;\; 1 \leq j \leq i,\\
\ell(y) +{\displaystyle  \sum_{k=i+1}^{j} (\xb_k - \yb_k)  +
\yb_{j}- x_{j}} &\mbox{for} \;\; i+1 \leq j \leq n-1,\\
\ell(y) +{\displaystyle  \sum_{k=i+1}^{n-1} (\xb_k - \yb_k) +
x_n }&\mbox{for} \;\; j=n,
\end{cases}
\\
\eta'_{i,j}(x,y) &=&
\begin{cases}
\ell(x)+ {\displaystyle  \sum_{k=1}^{i} (\yb_k -\xb_k)+
\sum_{k=1}^{j} (y_k -x_k)  + x_{j} -\yb_{j}}\;\;
\mbox{for}\;\;1 \leq j \leq n-1,\\
\ell(x)+ {\displaystyle 
\delta_{i,n-1} \left( \ell(x) - \ell(y) \right)+
 \sum_{k=1}^{i} (\yb_k -\xb_k)+
\sum_{k=1}^{n-1} (y_k -x_k)  -x_{n} }\\
\qquad 
\qquad \qquad \qquad \qquad
\qquad \qquad \qquad \qquad 
\qquad \mbox{for} \;\; j=n.
\end{cases}
\end{eqnarray*}

\begin{theorem}[Ref. \cite{KOTY1}, Theorem 4.28 and Remark 4.29]
\label{th:koty}
The image $y'\otimes x' = R(x\otimes y)$ of the combinatorial R is given by 
\begin{equation}\label{eq:combiR}
\begin{split}
& x_i' = x_i +V_{i-1}^* -V_i^*, \quad
\xb_i' = \xb_i +V_{i-1}^* +W_i -V_i^* -W_{i-1} 
\quad (1 \leq i \leq n-1),\\
& x_n' = x_n +V_{n-1}^* -V_{n-1},\quad y_n' = y_n +V_{n-1} -V_{n-1}^{*},\\
& y_i' = y_i +V_{i-1}+W_i -V_i -W_{i-1}, \quad
\yb_i' = \yb_i +V_{i-1} -V_i \quad (1 \leq i \leq n-1).
\end{split}
\end{equation}
Moreover, the local energy is given by 
\begin{equation}\label{HV}
H(x\otimes y)= V_0(x,y)
\end{equation} 
up to a constant shift.
\end{theorem}

The functions $V_1,\ldots, V_{n-1}, W_1, \ldots, W_{n-1}$
and $\sigma_1, \sigma_n$ and $\ast$ of them are relatives of 
the local energy.
In addition to the involutions 
$\sigma_1, \sigma_n$ and $\ast$, the combinatorial $R$ 
naturally acts on them by  
$(RV_0)(x,y) = V_0(y',x')$ with 
$y'\otimes x' = R(x\otimes y)$, etc.
Their transformation properties  
under $\sigma_1, \sigma_n, \ast$ and $R$ 
are summarized in Table \ref{tab:one} \cite{KOTY1}.
These involutions are commutative, thus for instance
$R(V^{\sigma_1}_0) = (R(V_0))^{\sigma_1} = V^{\sigma_1}_0$.

\begin{table}
	\caption{Transformation by 
$\sigma_1, \sigma_n, \ast$ and $R$.}
{\begin{tabular}[h]{|c|c|c|c|c|}
		\hline
		& $V_0$ & $V_i \, (1 \leq i \leq n-2)$ & $V_{n-1}$ & 
		$W_i \, (1 \leq i \leq n-1)$ \\
		\hline
	$\sigma_1$& $V_0^{\sigma_1}$ & $V_i$ & $V_{n-1}$ & $W_i$ \\
	$\sigma_n$& $V_0$ & $V_i$ & $V_{n-1}^*$ & $W_i$ \\
	$\ast$ & $V_0$ & $V_i^*$ & $V_{n-1}^*$ & 
	$W_i$ \\
	$R$ & $V_0$ & $W_i-V_i^*$& $V_{n-1}$& $W_i$ \\
		\hline
\end{tabular}}\label{tab:one}
\end{table}

\noindent
Due to these properties, 
there are a few simplifications in (\ref{eq:combiR}) as
\begin{eqnarray}
x_1' &=& x_1 +V_0 -V_1^*, \quad
\xb_1' = \xb_1 +V_0^{\sigma_1} -V_1^*, \nonumber\\
y_1' &=& y_1 +V_0^{\sigma_1} -V_1, \quad
\yb_1' = \yb_1 +V_0 -V_1. \label{yvv}
\end{eqnarray}

We write
\begin{equation}\label{ul}
u_l = (l,0,\ldots,0) \in B_l.
\end{equation}
By using Theorem \ref{th:koty}, 
one can show for any $\zeta \in B_m$ that
\begin{equation}\label{buu}
B_l \otimes B_m \ni u_l \otimes \zeta\;
\stackrel{\sim}{\mapsto}\;
u_m \otimes \xi' \in B_m \otimes B_l \;\;\text{if } l \ge m
\end{equation} 
for some $\xi'$ 
under the combinatorial $R$.
In particular 
\begin{equation*}
u_l \otimes u_m \simeq u_m \otimes u_l
\end{equation*}
holds.
The functions in Table \ref{tab:one} attain their maximum 
$V_0 = V_0^{\sigma_1} = V_i = V_i^* = l + m$ and
$W_i = 2(l+m)$ for $1 \leq i \leq n-1$ at 
$(x,y) = (u_l, u_m)$.

For $\xi \otimes \zeta \in B_l\otimes B_m$,
let $x\in B'_l$ and $y \in B'_m$ be the elements 
corresponding to $\xi$ and $\zeta$, respectively. 
We set 
\begin{eqnarray}
v_i(\xi\otimes \zeta)&=&\ell(\xi)+\ell(\zeta)
-V_i(x,y)\;\;(0 \le i \le n-1)\label{vi}\\
v^{\sigma_1}_0(\xi\otimes \zeta)&=&\ell(\xi)+\ell(\zeta)
-V^{\sigma_1}_0(x,y),\label{v01}\\
v^\ast_i(\xi\otimes \zeta)&=&\ell(\xi)+\ell(\zeta)-V^\ast_i(x,y)\;\;
(1 \le i \le n-1),\label{vis}\\
w_i(\xi\otimes \zeta) &=& 2\ell(\xi)+2\ell(\zeta)-W_i(x,y)\;\;
(1 \le i \le n-1),\label{wi}
\end{eqnarray}
and call them {\em generalized local energies}. 
Note that $w_{n-1}-v_{n-1}=v^{\ast}_{n-1}$.
They are building blocks of the piecewise linear 
formula of the combinatorial $R$ (\ref{eq:combiR}).
{}From the above remark, generalized local energies are all nonnegative and 
normalized so that 
\begin{equation}\label{zuu}
g(u_l\otimes u_m) = 0 \;\;
\text{for any } \;
g = v_i, v_0^{\sigma_1}, v_i^\ast \;\text{and}\;w_i.
\end{equation}

For $\zeta=(\zeta_1,\ldots,\zeta_n,\overline{\zeta}_n,\ldots,
\overline{\zeta}_1) \in B_l$, we introduce
\begin{eqnarray}
{\mathfrak a}(\zeta) &=& \zeta_2+\cdots \zeta_n+
\overline{\zeta}_n+\cdots +\overline{\zeta}_{2}
+2\overline{\zeta}_1
= \ell(\zeta)+\overline{\zeta}_1-\zeta_1, \label{all}\\
\gamma_{v_a}(\zeta) &=& \zeta_2+\cdots \zeta_n+
\overline{\zeta}_n+\cdots +\overline{\zeta}_{a+1}
+\overline{\zeta}_1
\;\;(0 \le a \le n-2),\label{gva}\\
\gamma_{v_{n-1}}(\zeta) &=& \zeta_2+\cdots 
+\zeta_{n-1} + \zeta_n+
\overline{\zeta}_1,\label{gvn1}\\
\gamma_{v^\ast_{n-1}}(\zeta) &=& \zeta_2+\cdots +\zeta_{n-1}
+\overline{\zeta}_n+
\overline{\zeta}_1,\label{gvn2}\\
\gamma_{w_a-v_a}(\zeta) &=&
\zeta_2+\cdots + \zeta_a+\overline{\zeta}_1\;\;(1 \le a \le n-2),
\label{gwv}\\
\gamma_{v^{\sigma_1}_0}(\zeta) &=& 0.\label{gvs}
\end{eqnarray}
Note that  $\gamma_{v_0}(\zeta) = \mathfrak{a}(\zeta)$.

\begin{lemma}\label{le:ato}
Let $\xi=(\xi_1,\ldots, \overline{\xi}_1) \in B_l$
and $\zeta=(\zeta_1,\ldots, \overline{\zeta}_1) \in B_m$.
Set $\zeta'\otimes \xi' = R(\xi\otimes \zeta)
\in B_m \otimes B_l$.
For $\xi_1$ (hence $l$ as well) sufficiently large, 
the relation
\begin{equation*}
g(\xi\otimes \zeta) = \gamma_g(\zeta)+\mathfrak{a}(\zeta')
-\gamma_g(\zeta')
\end{equation*}
is valid for $g$ appearing in (\ref{gva})--(\ref{gvs}),
where the left hand side with  
$g= w_a-v_a$ is to be understood as 
$w_a(\xi\otimes \zeta)-v_a(\xi\otimes \zeta)$. 
\end{lemma}

\begin{proof}
One can check that $\xi_1\ge m-\zeta_1+\overline{\zeta}_1$
is sufficient to guarantee that $V_0$ (\ref{Vi}) is equal to 
$\eta'_{0,1}= l +\zeta_1-\overline{\zeta}_1$.
This implies that 
$v_0(\xi\otimes \zeta) = m-\zeta_1+\overline{\zeta}_1
=\gamma_{v_0}(\zeta)$
showing the $g=v_0$ case.
All the other cases are deduced from this, 
(\ref{eq:combiR}) and (\ref{xz4})--(\ref{xz7})
without using the concrete forms of 
$\theta_{i,j}, \theta'_{i,j}, \eta_{i,j}$ and $\eta'_{i,j}$.
\end{proof}

For $g=w_a\, (1 \le a \le n-2)$, an analogue of 
Lemma \ref{le:ato} holds with 
\begin{eqnarray}
w_a(\xi\otimes \zeta) 
&=& \gamma_{w_a}(\zeta)+2\mathfrak{a}(\zeta')
-\gamma_{w_a}(\zeta'),\nonumber\\
\gamma_{w_a}(\zeta) &=& \gamma_{w_a-v_a}(\zeta)
+\gamma_{v_a}(\zeta) \nonumber\\
&=& \mathfrak{a}(\zeta) + 
\zeta_2+\cdots + \zeta_a 
-(\overline{\zeta}_a+\cdots + \overline{\zeta}_2).
\label{wga2}
\end{eqnarray}

As $\xi_1$ gets large, $\zeta'$ stabilizes since 
it is a piecewise linear function of $\xi_1$ staying in a finite set $B_m$.  
($\xi_1\ge m$ seems sufficient for the convergence.)
Therefore Lemma \ref{le:ato} ensures that all the generalized local energies 
$g(\xi\otimes \zeta)$ 
are well defined in the limit $\xi_1 \rightarrow \infty$.

\subsection{Generalized energies}\label{ss:gef}

For $p=p_1 \otimes \cdots \otimes p_L
\in B_{l_1}\otimes\cdots\otimes B_{l_L}$,
define $p^{(i)}_j \in B_{l_j}$ ($i<j$) by
\begin{eqnarray}\label{pij}
(B_{l_i}\otimes\cdots\otimes B_{l_{j-1}})\otimes B_{l_j}\;&\simarrow&
\;\;B_{l_j}\otimes (B_{l_i}\otimes\cdots\otimes B_{l_{j-1}})\nonumber\\
p_{i}\otimes\cdots\otimes p_{j-1}\otimes p_{j}
\;\;&\mapsto&\;\;p^{(i)}_j\otimes p'_i\otimes\cdots\otimes p'_{j-1},
\end{eqnarray}
sending $p_j$ to the left by successive applications of 
the combinatorial $R$.
We set $p^{(j)}_j = p_j$.
For any generalized local energy 
in (\ref{vi})--(\ref{wi}), we define the 
{\em generalized energy} of 
$p=p_1 \otimes \cdots \otimes p_L
\in B_{l_1}\otimes\cdots\otimes B_{l_L}$ by 
\begin{equation}\label{efp}
{\mathcal E}_g(p) = \sum_{0 \le i < j \le L}
g(p_i\otimes p^{(i+1)}_j)
\end{equation}
by taking $p_0 = u_l$ with sufficiently large $l$.
This is well defined (finite) due to Lemma \ref{le:ato}
and the comment following it. 
In the rest of the paper, we will simply write 
$p_0=u_\infty \in B_\infty$.

When $g=v_0$, (\ref{efp}) is the energy introduced in 
Refs. \cite{NY} and \cite{HKOTY} up to a sign and a constant shift.
If furthermore 
$l_1,\ldots, l_L$ are all equal, then $p^{(i+1)}_j = p_{i+1}$ holds 
and (\ref{efp}) reduces to 
\begin{equation*}
{\mathcal E}_{v_0}(p) = \sum_{0 \le  i < L}
(L-i)\,v_0(p_i\otimes p_{i+1}).
\end{equation*}
Its generating function 
$\sum_p q^{{\mathcal E}_{v_0}(p)}$ 
is a version of the one dimensional configuration sum
going back to Refs. \cite{Ba} and \cite{ABF}, which is the essential 
ingredient in the corner transfer matrix method. 

Any quantity $G(p_\alpha\otimes 
p_{\alpha+1}\otimes \cdots \otimes p_\beta)$
will be said {\em $R$-invariant} if 
$G(\cdots \otimes p_i\otimes p_{i+1}\otimes \cdots) = 
G(\cdots \otimes R(p_i\otimes p_{i+1})\otimes \cdots)$
for any $\alpha\le i <\beta$.

\begin{remark}\label{re:rinv}
Due to the transformation property 
under $R$ in Table \ref{tab:one}, 
the generalized energy 
${\mathcal E}_g(p)$ is $R$-invariant 
for $g = v_0, v^{\sigma_1}_0, v_{n-1}, v^\ast_{n-1}$
and $w_1,\ldots, w_{n-1}$.
On the other hand, 
${\mathcal E}_g$ with 
$g=v_1,\ldots, v_{n-2}$ and 
$v^\ast_1,\ldots, v^\ast_{n-2}$ are {\em not} $R$-invariant.
\end{remark}

Let us depict the relation $R(b\otimes c) = \tilde{c}\otimes\tilde{b}$ as
\vspace{0.2cm}
\begin{equation*}
\begin{picture}(20,30)(-13,-14)


\put(0,0){\line(1,-1){10}} \put(14,-18){${\tilde b}$}
\put(0,0){\line(1,1){10}} \put(12,12){${c}$}
\put(0,0){\line(-1,-1){10}} \put(-17,-18){${\tilde c}$}
\put(0,0){\line(-1,1){10}} \put(-16,12){$b$}

\end{picture}
\end{equation*}
Then the Yang-Baxter equation (\ref{ybe}) takes the 
well known form:
\begin{equation*}
\begin{picture}(80,40)(0,-2)
\put(0,0){\line(1,1){30}}
\put(30,0){\line(-1,1){30}}

\put(15,30){\line(-1,-1){12}}
\qbezier(3,18)(0,15)(3,12)
\put(15,0){\line(-1,1){12}}

\put(40,15){$=$}

\multiput(60,0)(0,0){1}{
\put(0,0){\line(1,1){30}}
\put(30,0){\line(-1,1){30}}

\put(15,30){\line(1,-1){12}}
\qbezier(27,18)(30,15)(27,12)
\put(15,0){\line(1,1){12}}
}

\end{picture}
\end{equation*}
The defining relation (\ref{pij}) of $p^{(i)}_j$ 
looks as 
\begin{equation}\label{pene}
\begin{picture}(100,70)(-50,-17)
\put(-5,-5){\line(1,1){45}}
\put(23,40){\line(1,-1){17}}
\put(12,29){\line(1,-1){17}}
\multiput(4,10)(2,2){3}{\put(0,0){$\cdot$}}
\put(-5,12){\line(1,-1){17}}
\put(8,45){$p_{j-1}$}\put(41,45){$p_j$}
\put(-8,33){$p_{j-2}$}
\put(-15,15){$p_i$}
\put(-17,-18){$p^{(i)}_j$}
\end{picture}
\end{equation}
Remember that each vertex is associated with various
generalized local energies $g(b\otimes c)$. 
Let
\begin{equation}\label{Ig}
I_g=I_g(p_i\otimes \cdots \otimes p_{j-1}\otimes p_j)
=\sum_{i\le k < j}g(p_k\otimes p^{(k+1)}_j)
\end{equation}
be the sum of generalized 
local energy $g$ over all the vertices in (\ref{pene}).
Then the generalized energy (\ref{efp}) is expressed as
\begin{eqnarray}\label{eig}
{\mathcal E}_g(p_1\otimes \cdots \otimes p_L)
&=& {\mathcal E}_g(p_1\otimes \cdots \otimes p_{L-1})
+I_g(u_\infty\otimes p_1\otimes \cdots \otimes p_L)\nonumber\\
&=& \sum_{1 \le j \le L}
I_g(u_\infty\otimes p_1\otimes \cdots \otimes p_{j-1}\otimes p_j).
\end{eqnarray}
{}From (\ref{buu}) and (\ref{zuu}),  it follows that
\begin{eqnarray}
I_g(u_\infty\otimes u_\infty 
\otimes p_1\otimes \cdots \otimes p_j)
&=& I_g(u_\infty\otimes p_1\otimes \cdots \otimes p_j),
\label{uuI}\\
{\mathcal E}_g(u_\infty\otimes p_1\otimes \cdots \otimes p_L)
&=& {\mathcal E}_g(p_1\otimes \cdots \otimes p_L).
\label{uE}
\end{eqnarray}

\begin{lemma}\label{le:chiral}
In (\ref{pene}), the following quantities 
are $R$-invariant as 
the functions of 
$p_i \otimes \cdots \otimes p_{j-1}$.
(i) The element $p^{(i)}_j$.
(ii) $I_g$ (\ref{Ig}) for any 
$g=v_a\,(0\le a \le n-1), \,w_a-v_a\,(1 \le a \le n-2)$,
$v^\ast_{n-1}$ and $v^{\sigma_1}_0$.
\end{lemma}

\begin{proof}
(i) This is due to the classical part of the 
Yang-Baxter equation (\ref{ybe}).
(ii) The $R$-invariance of $I_{v_0}$
follows from (\ref{vi}), (\ref{HV}) and the 
affine part of the Yang-Baxter equation.
Let $y, y' \in B'_{l_j}$ be the elements
corresponding to $p_j, p^{(i)}_j \in B_{l_j}$, respectively.
{}From (\ref{yvv}), we have
$\overline{y}'_1-\overline{y}_1=I_{v_1}-I_{v_0}$.
Since the left hand side is $R$-invariant by (i), 
this relation implies the $R$-invariance of $I_{v_1}$. 
By similarly using the $R$-invariance of 
$\overline{y}'_i-\overline{y}_i$ and $y'_i-y_i$ 
in (\ref{yvv}) and (\ref{eq:combiR}), 
one can verify the $R$-invariance of the other $I_g$.   
\end{proof}

We note that Lemma \ref{le:chiral} is applicable 
to the situation $i=0$, i.e., 
$p_i\otimes \cdots \otimes p_{j-1}
= u_\infty \otimes p_1 \otimes \cdots \otimes p_{j-1}$.

\begin{remark}\label{re:chiral}
Lemma \ref{le:chiral} (ii) does not concern    
$I_{v^\ast_1},\ldots, I_{v^\ast_{n-1}}$. 
In fact the proof does not persist since   
$V^\ast_1, \ldots, V^\ast_{n-2}$ are not contained in 
any difference of the components of $y'$ and $y$ in (\ref{eq:combiR}).
Similarly, $V_1, \ldots, V_{n-2}$ do not appear
in the differences of $x$ and $x'$.
This {\em chirality} 
of the combinatorial $R$ is a characteristic feature of 
the $D^{(1)}_n$ case. 
In contrast,  
$I_g$'s with all the generalized local energies $g$ for $A^{(1)}_n$
(so called $i$th (un)winding number \cite{KSY}) are $R$-invariant.
We shall come back to this point again in Section \ref{ss:cp}.
\end{remark}

The following proposition 
and its proof are parallel with 
Lemma 4.4 in Ref. \cite{KSY} for
type $A^{(1)}_n$.

\begin{proposition}\label{pr:zushi}
For $g = v_i \,(0 \le i \le n-1), 
w_i \,(1\le i \le n-1), \,
v^{\sigma_1}_0$ and $v^\ast_{n-1}$,
the generalized energy 
${\mathcal E}_g(p_1\otimes \cdots \otimes p_L)$
(\ref{efp}) is equal to the sum of the 
generalized local energy $g$ attached to all the 
vertices in the following diagram ($L=3$ example):

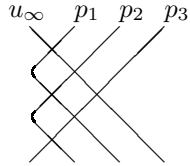
\begin{picture}(80,65)(-130,5)
\put(-8,55){$u_\infty$}
\put(17,55){$p_1$}
\put(34,55){$p_2$}
\put(51,55){$p_3$}

\put(0,0){\line(1,1){51}}
\put(34,51){\line(-1,-1){32}}
\put(17,51){\line(-1,-1){15}}
\qbezier(2,19)(0,17)(2,15)
\qbezier(2,32)(0,34)(2,36)
\put(17,0){\line(-1,1){15}}
\put(34,0){\line(-1,1){32}}
\put(51,0){\line(-1,1){51}}

\end{picture}
\end{proposition}

\begin{proof}
We invoke the induction on $L$.
For $L=1$, one has 
${\mathcal E}_g(p_1) = g(u_\infty\otimes p_1)$,
and the assertion is obviously true.
We illustrate the induction step from $L=2$ to $L=3$.
Consider the following identity obtained by 
successive applications of the Yang-Baxter equation:

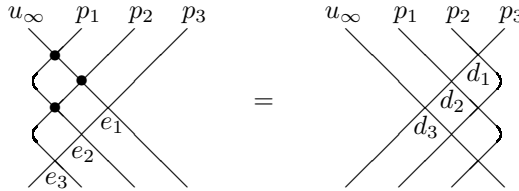
\begin{equation*}
\begin{picture}(200,60)(-30,0)

\put(-8,64){$u_\infty$}
\put(18,64){$p_1$}
\put(38,64){$p_2$}
\put(58,64){$p_3$}

\put(0,0){\line(1,1){60}}
\put(0,60){\line(1,-1){60}}

\put(20,60){\line(-1,-1){17}}
\qbezier(3,43)(0,40)(3,37)
\put(40,60){\line(-1,-1){37}}

\put(20,0){\line(-1,1){17}}
\qbezier(3,23)(0,20)(3,17)
\put(40,0){\line(-1,1){37}}

\put(7.5,47.5){$\bullet$}\put(17.5,38){$\bullet$}
\put(7.5,27.5){$\bullet$}

\put(27,22){$e_1$}
\put(16,11){${\scriptsize e_2}$}
\put(6,1){$e_3$}

\put(85,30){$=$}

\multiput(120,0)(0,0){1}{
\put(-8,64){$u_\infty$}
\put(18,64){$p_1$}
\put(38,64){$p_2$}
\put(58,64){$p_3$}

\put(0,0){\line(1,1){60}}
\put(0,60){\line(1,-1){60}}

\put(20,60){\line(1,-1){37}}
\qbezier(57,43)(60,40)(57,37)
\put(40,60){\line(1,-1){17}}

\put(20,0){\line(1,1){37}}
\qbezier(57,23)(60,20)(57,17)
\put(40,0){\line(1,1){17}}

\put(45.5,40){$d_1$}
\put(35.5,30){$d_2$}
\put(25.5, 20){$d_3$}

}

\end{picture}
\end{equation*}
Here $\bullet, e_i, d_i$ stand for 
the values of $g$ at the attached vertices.
By the induction assumption, 
the sum of the three $\bullet$ in the left hand side is equal to 
${\mathcal E}_g(p_1\otimes p_2)$.
In view of the recursion relation (\ref{eig}), we are to verify
$e_1+e_2+e_3=I_g(u_\infty\otimes p_1\otimes p_2\otimes p_3)$.
By the definition, 
$I_g(u_\infty\otimes p_1\otimes p_2\otimes p_3)=d_1+d_2+d_3$ 
in the right diagram,
where $d_1= g(p_2\otimes p^{(3)}_3),
d_2= g(p_1\otimes p^{(2)}_3), 
d_3= g(u_\infty\otimes p^{(1)}_3)$.
Thanks to the $R$-invariance of $I_g$ in 
Lemma \ref{le:chiral} (ii), this is equal to 
$e_1+e_2+e_3$.
\end{proof}

\section{Integrable $D^{(1)}_n$ cellular automaton}\label{sec:sca}

\subsection{States and time evolution}
Let us recall the integrable $D^{(1)}_n$ cellular automaton 
associated with $B_l$ \cite{HKT1,HKT2}.
Consider the crystal 
$B_{l_1}\otimes  \cdots \otimes B_{l_L}$.
Its elements are called states.
We regard each component $(\zeta_1, \ldots, \overline{\zeta}_1) \in B_l$  
as a capacity $l$ box containing $\zeta_a$ particles $a$ and 
$\overline{\zeta}_a$ anti-particles $\overline{a}$ for $2 \le a \le n$, 
and furthermore 
$\overline{\zeta}_1$ extra $\overline{1}$'s which we call  bound pairs.
The remaining $1$'s represent empty space in the box.
Thus $u_l$ stands for an empty box.
The indices $a, \overline{a}$ 
will be referred as color of particles and anti-particles, respectively.
A state 
$p_1 \otimes \cdots \otimes p_L
\in B_{l_1}\otimes\cdots\otimes B_{l_L}$
represents a configuration of particles, 
anti-particles and bound pairs in an array of boxes 
with capacity $l_1,\ldots, l_L$.
Because of the constraint $\zeta_n\overline{\zeta}_n=0$ in 
(\ref{Bl}), particles $n$ and anti-particles 
$\overline{n}$ do not coexist within a box.
We shall denote the element $(3,0,1,0,2,0,1,0) \in B_7$ of $D^{(1)}_4$, 
for example, by 
$1113\overline{4}\overline{4}\overline{2}$, etc.

For a positive integer $l$, 
we define the time evolution
$T_l(p)= p'_1 \otimes \cdots \otimes p'_L$ of 
a state $p = p_1 \otimes \cdots \otimes p_L$ by
\begin{equation*}
u_l \otimes p_1\otimes \cdots 
\otimes p_L\simeq 
p'_1\otimes \cdots \otimes p'_L\otimes 
\xi
\end{equation*}
under the isomorphism 
$B_l \otimes 
(B_{l_1} \otimes \cdots \otimes B_{l_L})  
\simeq (B_{l_1} \otimes \cdots \otimes 
B_{l_L})\otimes B_l$.
Here $\xi \in B_l$ as well as $T_l(p)$ are uniquely determined 
from $p$ by the combinatorial $R$.
It can be shown that
\begin{equation}\label{eq:uka}
\xi = u_l\; \hbox{ if 
$p_j = u_{l_j}$ for $L' \le j \le L$ with 
sufficiently large $L-L'$}.
\end{equation}
The time evolutions $\{T_l\}$ form a commuting family, and 
$T_l$ stabilizes as $l$ gets large,  
which will be denoted by $T_\infty$.

When $l_1=\cdots =l_L=1$, $T_\infty$ is factorized as 
\begin{equation*}
T_\infty = K_{2} K_{3}\cdots K_{n}K_{\overline{n}}\cdots
K_{\overline{3}}K_{\overline{2}}
\end{equation*}
with $K_a$ given by the following algorithm 
(we understand $\overline{\overline{a}}=a$) \cite{HKT2}.
\begin{enumerate}
\item Replace each $\overline{1}$ by a pair 
$a, \overline{a}$ within a box.
\item Pick the leftmost $a$ (if any) and move it to the nearest right 
box which is empty or containing just $\overline{a}$.
(Boxes involving the pair $a, \overline{a}$ are prohibited as the destination.)

\item  Repeat (2) for those $a$'s that are not yet moved
    until all of $a$'s are moved once.
\item Replace the pair $a, \overline{a}$ within a box (if any) by 
$\overline{1}$.
\end{enumerate}

In the above, taking some $b \, (\neq 1) \in B_1$ 
away from a box means the change 
of the local state $b \rightarrow 1$.
Similarly, putting $b \, (\neq 1) \in B_1$ into an empty box means 
the change $1 \rightarrow b$.
The steps (1) and (4) can be viewed as pair creation and 
annihilation, respectively.

\begin{example}\label{ex:hatten1}
We consider $D^{(1)}_4$ and states from $B^{\otimes L}_1$ with $L=48$.
Successive time evolutions of the initial state 
on the first line under $T_\infty$ is presented downward,
where $b_1\otimes \cdots \otimes b_L$ is simply denoted by 
a horizontal  array $b_1 b_2 \ldots b_L$.
It shows a collision of solitons with amplitudes $6$ and $3$.
\begin{center}
$1\,1\,1\,\overline{3}\,\overline{3}\,\overline{4}\,3\,2\,2\,1\,1\,1\,1\,1\,\overline{2}\,\overline{2}\,\overline{4}\,1\,1\,1\,1\,1\,1\,1\,1\,1\,1\,1\,1\,1\,1\,1\,1\,1\,1\,1\,1\,1\,1\,1\,1\,1\,1\,1\,1\,1\,1\,1$\\
$1\,1\,1\,1\,1\,1\,1\,1\,1\,\overline{3}\,\overline{3}\,\overline{4}\,3\,2\,2\,1\,1\,\overline{2}\,\overline{2}\,\overline{4}\,1\,1\,1\,1\,1\,1\,1\,1\,1\,1\,1\,1\,1\,1\,1\,1\,1\,1\,1\,1\,1\,1\,1\,1\,1\,1\,1\,1$\\
$1\,1\,1\,1\,1\,1\,1\,1\,1\,1\,1\,1\,1\,1\,1\,\overline{3}\,\overline{3}\,\overline{4}\,3\,2\,\overline{1}\,\overline{2}\,\overline{4}\,1\,1\,1\,1\,1\,1\,1\,1\,1\,1\,1\,1\,1\,1\,1\,1\,1\,1\,1\,1\,1\,1\,1\,1\,1$\\
$1\,1\,1\,1\,1\,1\,1\,1\,1\,1\,1\,1\,1\,1\,1\,1\,1\,1\,1\,1\,1\,\overline{3}\,4\,\overline{1}\,\overline{1}\,\overline{4}\,\overline{4}\,\overline{4}\,1\,1\,1\,1\,1\,1\,1\,1\,1\,1\,1\,1\,1\,1\,1\,1\,1\,1\,1\,1$\\
$1\,1\,1\,1\,1\,1\,1\,1\,1\,1\,1\,1\,1\,1\,1\,1\,1\,1\,1\,1\,1\,1\,1\,1\,1\,4\,2\,2\,\overline{2}\,\overline{2}\,\overline{3}\,\overline{4}\,\overline{4}\,\overline{4}\,1\,1\,1\,1\,1\,1\,1\,1\,1\,1\,1\,1\,1\,1$\\
$1\,1\,1\,1\,1\,1\,1\,1\,1\,1\,1\,1\,1\,1\,1\,1\,1\,1\,1\,1\,1\,1\,1\,1\,1\,1\,1\,1\,4\,2\,2\,1\,1\,1\,\overline{2}\,\overline{2}\,\overline{3}\,\overline{4}\,\overline{4}\,\overline{4}\,1\,1\,1\,1\,1\,1\,1\,1$\\
$1\,1\,1\,1\,1\,1\,1\,1\,1\,1\,1\,1\,1\,1\,1\,1\,1\,1\,1\,1\,1\,1\,1\,1\,1\,1\,1\,1\,1\,1\,1\,4\,2\,2\,1\,1\,1\,1\,1\,1\,\overline{2}\,\overline{2}\,\overline{3}\,\overline{4}\,\overline{4}\,\overline{4}\,1\,1$
\end{center}
\end{example}

For $l_1,\ldots, l_L$ general, $T_\infty$ still admits 
a similar, although slightly more involved, algorithm.
We omit it here and give an example instead.  

\begin{example}\label{ex:hatten2}
We consider $D^{(1)}_4$ and states from 
$B_6\otimes B_3 \otimes B_4 \otimes B_4 \otimes B^{\otimes 8}_2$.
\[
\begin{array}{llllllllllll}
124\overline{3}\overline{2}\overline{1}\;\cdot& 
234 \;\cdot& 2\overline{3}\overline{2}\overline{1}
\;\cdot& 13\overline{4}\overline{4}\;\cdot& 11 \;\cdot& 
11 \;\cdot& 11 \;\cdot& 11 \;\cdot& 11 \;\cdot& 
11 \;\cdot& 11 \;\cdot& 11 \\
 111111 \;\cdot& 123 \;\cdot& 
12\overline{3}\overline{2} \;\cdot& 
344\overline{1}\;\cdot& \overline{3}\overline{2} \;\cdot
& \overline{4}\overline{3} \;\cdot&
3\overline{4}\;\cdot& 12 \;\cdot& 11 \;\cdot& 11 \;\cdot& 
11 \;\cdot& 11 \\
 111111 \;\cdot& 111 \;\cdot& 1123 \;\cdot& 1113 \;\cdot& 
\overline{3}\overline{3} \;\cdot& 44 \;\cdot&
23 \;\cdot& 1\overline{2}\;\cdot& 
\overline{3}\overline{2} \;\cdot& \overline{4}\overline{3} 
\;\cdot& 3\overline{4} \;\cdot& 12 \\
111111 \;\cdot& 111 \;\cdot& 1111 \;\cdot& 1123 \;\cdot& 
13 \;\cdot& 11 \;\cdot&
11 \;\cdot& \overline{3}\overline{3} \;\cdot& 44 \;\cdot& 
23 \;\cdot& 11 \;\cdot& 1\overline{2}\\
 111111 \;\cdot& 111 \;\cdot& 1111 \;\cdot& 1111 \;\cdot& 
23 \;\cdot& 13 \;\cdot&
   11 \;\cdot& 11 \;\cdot& 11 \;\cdot& 11 \;\cdot& 
\overline{3}\overline{3} \;\cdot& 44 \\
 111111 \;\cdot& 111 \;\cdot& 1111 \;\cdot& 1111 \;\cdot& 
11 \;\cdot& 23 \;\cdot&
   13 \;\cdot& 11 \;\cdot& 11 \;\cdot& 11 \;\cdot& 11 \;\cdot& 11 \\
 111111 \;\cdot& 111 \;\cdot& 1111 \;\cdot& 
1111 \;\cdot& 11 \;\cdot& 11 \;\cdot&
   23 \;\cdot& 13 \;\cdot& 11 \;\cdot& 11 \;\cdot& 
11 \;\cdot& 11 \\
 111111 \;\cdot& 111 \;\cdot& 1111 \;\cdot& 1111 \;\cdot& 
11 \;\cdot& 11 \;\cdot&
   11 \;\cdot& 23 \;\cdot& 13 \;\cdot& 11 \;\cdot& 
11 \;\cdot& 11 \\
 111111 \;\cdot& 111 \;\cdot& 1111 \;\cdot& 1111 \;\cdot& 
11 \;\cdot& 11 \;\cdot&
   11 \;\cdot& 11 \;\cdot& 23 \;\cdot& 13 \;\cdot& 
11 \;\cdot& 11 \\
 111111 \;\cdot& 111 \;\cdot& 1111 \;\cdot& 
1111 \;\cdot& 11 \;\cdot& 11 \;\cdot&
   11 \;\cdot& 11 \;\cdot& 11 \;\cdot& 23 \;\cdot& 
13 \;\cdot& 11 \\
 111111 \;\cdot& 111 \;\cdot& 1111 \;\cdot& 
1111 \;\cdot& 11 \;\cdot& 11 \;\cdot&
   11 \;\cdot& 11 \;\cdot& 11 \;\cdot& 11 \;\cdot& 23 \;\cdot& 13
\end{array}
\]
Here, $\cdot$ represents $\otimes$.
\end{example}

\begin{remark}\label{re:u}
Suppose 
$p_j=u_{l_j}$ for $1 \le j \le k$ in a state 
$p=p_1 \otimes \cdots \otimes p_L$.
Then in the state 
$T_\infty(p)=p'_1 \otimes \cdots \otimes p'_L$,
$p'_j=u_{l_j}$ is valid for $1 \le j \le k+1$.
\end{remark}  

We postpone the inverse scattering formalism 
of the dynamics to Theorem \ref{th:ist}.

\subsection{Counting particles and anti-particles} 
Recall that $\mathfrak{a}(\zeta)$ is defined in (\ref{all}).
In our present context, 
it is the number of all the particles and anti-particles
within a box specified by $\zeta \in B_l$, 
where the term $2\overline{\zeta}_1$
means that a bound pair is regarded as
a pair of a particle and an anti-particle (whose color is unspecified). 
The symbol $\mathfrak{a}$ means $\mathfrak{a}$ll kinds of 
(anti-)particles.

Let $p=p_1 \otimes \cdots \otimes p_L$ be a state
and write its time evolution as
\[T^t_\infty(p_1 \otimes \cdots \otimes p_L)
= p^t_1 \otimes \cdots \otimes p^t_L,
\]
where $p^t_j \in B_{l_j}$.
We write 
$p_j =p^0_j=(\zeta_{j,1},\ldots,  \zeta_{j,n},
\overline{\zeta}_{j,n}, \ldots, \overline{\zeta}_{j,1})
\in B_{l_j}$.
For any elements $a_1,\ldots, a_r$ of 
$\{2,3,\ldots,n,\overline{n},\ldots, \overline{2},\overline{1}\}$,
we define the {\em counting function}
\begin{equation}\label{eq:rho}
\rho_{a_1,\ldots, a_r}(p) = \sum_{j=1}^L
(\zeta_{j,a_1}+\dots+\zeta_{j,a_r})
+ \sum_{t\ge 1}\sum_{j=1}^L\mathfrak{a}(p^t_j),
\end{equation}
where 
$\zeta_{j, \overline{3}}=\overline{\zeta}_{j,3}$, etc.
The dependence on $a_1,\ldots, a_r$ enters the first term only.
The indices in $\rho_{a_1,\ldots, a_r}$ will always be arranged 
in the order $2,3,\ldots,n,\overline{n},\ldots, 
\overline{3},\overline{2},\overline{1}$.
The second term is finite due to Remark \ref{re:u}.
In fact the double sum may well be restricted to
$\sum_{t=1}^{L-1}\sum_{j=t+1}^L$ where 
the nonzero contributions are contained.
This region is depicted as the SW quadrant of 
the time evolution patterns like 
Example \ref{ex:hatten1} and \ref{ex:hatten2}. 
\begin{equation}\label{sw}
\unitlength 0.1in
\begin{picture}( 12.0000,  12.0000)( 10.0000, -17.0000)
\special{pn 8}%
\special{pa 988 600}%
\special{pa 1948 600}%
\special{pa 1948 760}%
\special{pa 988 760}%
\special{pa 988 600}%
\special{fp}%
%
\special{pn 8}%
\special{sh 0.300}%
\special{pa 988 760}%
\special{pa 1948 760}%
\special{pa 1948 1720}%
\special{pa 988 760}%
\special{ip}%
\put(18.4300,-6.8800){\makebox(0,0){$p_L$}}%
\put(13.1500,-6.8800){\makebox(0,0){$p_2$}}%
\put(11.1500,-6.8800){\makebox(0,0){$p_1$}}%
\put(15.6300,-6.8800){\makebox(0,0){$\cdots$}}%
\end{picture}
\end{equation}
The first term in (\ref{eq:rho}) is the number of (anti-)particles
with colors $a_1,\ldots, a_r$ 
contained in the top row which is the state $p$ itself.
The second term counts all kinds of 
particles and anti-particles in the hatched domain 
in (\ref{sw}).\footnote{
More precisely, it should be hatched in a
staircase shape.}
By the definition it follows that 
\begin{equation}\label{nashi}
\rho_\emptyset(p) = 
\rho_{2,\ldots,n,\overline{n},\ldots,\overline{2},\overline{1},\overline{1}}
(T_\infty(p)).
\end{equation}

Given a state $p = p_1\otimes \cdots \otimes p_L$, we write  
\begin{equation}\label{pk}
p_{[k]} = p_1\otimes \cdots \otimes p_k\quad(1 \le k \le L).
\end{equation}

\begin{example}\label{ex:soba}
Let $p$ be the state in the first line in Example \ref{ex:hatten2}.
Then, the counting function $\rho_{a_1,\ldots, a_r}(p_{[k]})$ 
for $1\le k \le 9$ 
takes the following values. (The middle column shows 
$g$ such that $\rho_{a_1,\ldots, a_r} = \rho_g$ in 
(\ref{rr1})--(\ref{rr4}).)
\begin{equation*}
\begin{array}{|c|c|rcccccccc|}
\hline
a_1,\ldots, a_r & g &k=1 \;& \, 2\; & \,3\; & \,4\; 
& 5 & 6 & 7 & 8 & 9\\ 
\hline
&&&&&&&&&&\vspace{-0.3cm}\\
234\overline{4}\overline{3}\overline{2}\overline{1}
\overline{1}
& v_0 & \;
6& 11& 21& 32& 39& 46& 53& 60& 67\\
234\overline{4}\overline{3}\overline{2}\overline{1} 
& v_1 &\;
5& 10& 19& 30& 37& 44& 51& 58& 65 \\
234\overline{4}\overline{3}\overline{1} 
& v_2 &\;
4& 9& 17& 28& 35& 42& 49& 56& 63 \\
234\overline{1} 
& v_3 &\;
3& 8& 15& 24& 31& 38& 45& 52& 59\\
23\overline{4}\overline{1} 
& v^\ast_3 &\;
2& 6& 13& 24& 31& 38& 45& 52& 59\\
2\overline{1} 
& w_2-v_2 &\;
2& 5& 12& 20& 27& 34& 41& 48& 55\\
\overline{1} 
& \;w_1-v_1\; &\;
1& 3& 9& 17& 24& 31& 38& 45& 52\\
\emptyset 
& v^{\sigma_1}_0 &\;
0& 2& 7& 15& 22& 29& 36& 43& 50\\
\hline
\end{array}
\end{equation*}
\end{example}

\vspace{0.3cm}
We set $\rho_{a_1,\ldots, a_r}(p_{[0]})=0$
for any $a_1,\ldots, a_r$.
Consider the difference
$\rho_{2\overline{1}}(p_{[k]})- \rho_{\overline{1}}(p_{[k]})$
for example.
By the definition (\ref{eq:rho}),
it is the number of color $2$ particles contained in $p_{[k]}$.
Thus we have
\begin{eqnarray*}
&&\sharp \text{(color 2 particles) in } p_k \,(=\zeta_{k,2}) \\
&&\; = \rho_{2\overline{1}}(p_{[k]})- \rho_{\overline{1}}(p_{[k]})
-\rho_{2\overline{1}}(p_{[k-1]})+ \rho_{\overline{1}}(p_{[k-1]}).
\end{eqnarray*}
This is an example of the relations that reproduces 
a local variable from non-local counting functions.
Given $l_k$, the set of counting functions that are necessary and 
sufficient to completely reproduce 
the local state $p_k \in B_{l_k}$ is not unique.
However there is a choice that is linked with the 
generalized energies in Section \ref{ss:gef}.
By using the function $\gamma_g$ in 
(\ref{gva})--(\ref{gvs}), we set
\begin{equation}\label{rhog}
\rho_g(p)=\sum_{j=1}^L\gamma_g(p_j)
+ \sum_{t\ge 1}\sum_{j=1}^L\mathfrak{a}(p^t_j)
\end{equation}
for $g=v_a\,(0\le a \le n-1), \,w_a-v_a\,(1 \le a \le n-2)$,
$v^\ast_{n-1}$ and $v^{\sigma_1}_0$.
Although the notations $\rho_g$ here and $\rho_{a_1,\ldots, a_r}$ in 
(\ref{eq:rho}) are somewhat confusing, 
we dare to use the both in the sequel 
supposing the resemblance is not too serious.
Then (\ref{rhog}) is explicitly given as follows:
\begin{eqnarray}
\rho_{v_a}(p) &=& \rho_{2,\ldots,n,\overline{n},
\ldots,\overline{a+1}, \overline{1}}(p)\;\;(0 \le a \le n-2),
\label{rr1}\\
\rho_{v_{n-1}}(p) 
&=& \rho_{2,3,\ldots, n-1,n,\overline{1}}(p),\quad
\rho_{v^\ast_{n-1}}(p)
= \rho_{2,3,\ldots, n\!-\!1,\overline{n},\overline{1}}(p),
\label{rr2}\\
\rho_{w_a-v_a}(p) &=& \rho_{2,3,\ldots, a,\overline{1}}(p)
\;\;(1 \le a \le n-2),
\label{rr3}\\
\rho_{v^{\sigma_1}_0}(p)&=& \rho_{\emptyset}(p).
\label{rr4}
\end{eqnarray}
The last one is subsidiary in that 
$
\rho_{v^{\sigma_1}_0}(p)
= \rho_{w_1-v_1}(p)-\rho_{v_0}(p) +\rho_{v_1}(p)
$
holds reflecting (\ref{W1}).
One may also additionally introduce 
\begin{equation*}
\rho_{w_a}(p) = \rho_{w_a-v_a}(p) + \rho_{v_a}(p)
=  \sum_{j=1}^L\gamma_{w_a}(p_j)
+ 2\sum_{t\ge 1}\sum_{j=1}^L\mathfrak{a}(p^t_j)
\end{equation*}
for $1 \le a \le n-2$. See (\ref{wga2}).
For $D^{(1)}_4$,  the counting functions 
(\ref{rr1})--(\ref{rr4}) 
are precisely those listed in Example \ref{ex:soba}.

\begin{theorem}\label{pr:rho}
For $p_{[k]}$ in (\ref{pk}), set 
$\delta\rho_g = \rho_g(p_{[k]})-\rho_g(p_{[k-1]})$.
The counting functions (\ref{rr1})--(\ref{rr3}) 
reproduce the local state
$p_k = (\zeta_{1},\ldots,  \zeta_{n},
\overline{\zeta}_{n}, \ldots, \overline{\zeta}_{1})
\in B_{l_k}$ by
\begin{eqnarray*}
\zeta_{1} &=& l_k -\delta\rho_{v_0}+\delta\rho_{w_1-v_1},\\
\zeta_{a} &=& \delta\rho_{w_a-v_a}-\delta\rho_{w_{a-1}-v_{a-1}}
\;\;(2 \le a \le n-2),\\
\zeta_{n-1} &=& 
\min(\delta\rho_{v_{n-1}},\delta\rho_{v^\ast_{n-1}})
-\delta\rho_{w_{n-2}-v_{n-2}},\\
\zeta_{n} &=& \max(\delta\rho_{v_{n-1}}-\delta\rho_{v^\ast_{n-1}},0),\\
\overline{\zeta}_{n} &=& 
\max(\delta\rho_{v^\ast_{n-1}}-\delta\rho_{v_{n-1}},0),\\
\overline{\zeta}_{n-1} &=& 
-\max(\delta\rho_{v_{n-1}},\delta\rho_{v^\ast_{n-1}})
+\delta\rho_{v_{n-2}},\\
\overline{\zeta}_{a} &=& 
\delta\rho_{v_{a-1}}-\delta\rho_{v_a}
\;\;(1 \le a \le n-2).
\end{eqnarray*}
\end{theorem}
\begin{proof}
Straightforward by using (\ref{eq:rho}), (\ref{rr1})--(\ref{rr3}) and 
$\min(\zeta_{n}, \overline{\zeta}_{n})=0$.
\end{proof}

\section{Main result}\label{sec:main}

\subsection{Counting functions and generalized energies}
\begin{theorem}\label{th:main}
For any state $p \in B_{l_1} \otimes \cdots \otimes B_{l_L}$,
the counting functions and 
the generalized energies (\ref{efp}) coincide, namely,  
\begin{equation}\label{eq:main}
{\mathcal E}_g(p) = \rho_g(p)
\end{equation}
for $g=v_a\,(0\le a \le n-1), \,w_a-v_a\,(1 \le a \le n-2)$,
$v^\ast_{n-1}$ and $v^{\sigma_1}_0$.
\end{theorem}
Here, ${\mathcal E}_{w_a-v_a}(p)$ should be understood as
${\mathcal E}_{w_a}(p)-{\mathcal E}_{v_a}(p)$,
and the same convention is assumed for  $I_{w_a-v_a}(p)$
in the sequel.
Of course ${\mathcal E}_{w_a}(p) = \rho_{w_a}(p)$ follows as a corollary. 
The $g$'s in Theorem \ref{th:main} are the same as those considered 
in Lemma \ref{le:chiral} (ii) and (\ref{rhog}).
By substituting (\ref{rr1})--(\ref{rr4})
into (\ref{eq:main}), the theorem may be rephrased as
\begin{eqnarray*}
{\mathcal E}_{v_a}(p)
&=&\rho_{2,\ldots,n,\overline{n},
\ldots,\overline{a+1}, \overline{1}}(p)
\;\;\;(0 \le a \le n-2), \\
{\mathcal E}_{v_{n-1}}(p)&=&
\rho_{2,3,\ldots, n-1,n,\overline{1}}(p),\quad
{\mathcal E}_{v^\ast_{n-1}}(p)
=\rho_{2,3,\ldots, n\!-\!1,\overline{n},\overline{1}}(p),\\
{\mathcal E}_{w_a-v_a}(p)&=&
\rho_{2,3,\ldots, a,\overline{1}}(p)
\;\;\;(1 \le a \le n-2),\\
{\mathcal E}_{v^{\sigma_1}_0}(p) &=& 
\rho_\emptyset(p).
\end{eqnarray*}

For the proof we need one more Lemma.

\begin{lemma}\label{le:edif}
Let $p=p_1 \otimes \cdots \otimes p_L
\in B_{l_1} \otimes \cdots \otimes B_{l_L}$.
For those $g$'s in Theorem \ref{th:main}, 
the following equality is valid: 
\begin{equation}\label{edif}
{\mathcal E}_g(p) - {\mathcal E}_g(T_\infty(p))
= \sum_{i=1}^Lg(\xi^{(i)}\otimes p_i),
\end{equation}
where $\xi^{(i)}\in B_\infty$ is defined by 
$u_\infty\otimes p_1\otimes \cdots \otimes p_{i-1}
\simeq p'_1\otimes \cdots \otimes p'_{i-1}\otimes 
\xi^{(i)}$. 
\end{lemma}

\begin{proof}
The following proof simplifies the one for Proposition 4.6
in Ref. \cite{KSY} in that the assumption 
$l_1\ge \cdots \ge l_L$ is not needed.
We illustrate it for $L=3$.
Set $p = p_1\otimes p_2\otimes p_3$ and 
$T_\infty(p) = p'_1\otimes p'_2\otimes p'_3$.
Then, Proposition \ref{pr:zushi} tells that 
${\mathcal E}_g(T_\infty(p))$ 
is the sum of $g$ at all $\bullet$ in the following diagram. 
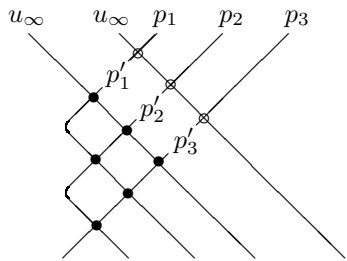
\begin{equation*}
\begin{picture}(150,100)(-30,0)
\unitlength 0.25mm

\put(-29,126){$u_\infty$}
\put(16,126){$u_\infty$}
\put(48,126){$p_1$}
\put(83,126){$p_2$}
\put(118,126){$p_3$}

\put(13,82){$\bullet$}
\put(14,49){$\bullet$}
\put(14.5,14){$\bullet$}
\put(30.5,65){$\bullet$}
\put(31,31){$\bullet$}
\put(47.5,48){$\bullet$}

\put(36.5,106){$\circ$}
\put(54,89){$\circ$}
\put(71.5,71){$\circ$}

\put(30,120){\line(1,-1){120}}
\put(-18,120){\line(1,-1){120}}
\put(59,59){$p'_3$}
\put(120,120){\line(-1,-1){49}}
\put(0,0){\line(1,1){56}}

\put(85,120){\line(-1,-1){32}}
\put(41,76){$p'_2$}
\put(3,38){\line(1,1){36}}
\qbezier(3,32)(0,35)(3,38)
\put(35,0){\line(-1,1){32}}

\put(50,120){\line(-1,-1){15}}
\put(23,94){$p'_1$}
\put(3,73){\line(1,1){18}}
\qbezier(3,67)(0,70)(3,73)
\put(70,0){\line(-1,1){67}}
\end{picture}
\end{equation*}
On the other hand, the right hand side of 
(\ref{edif}) is equal to the sum of $g$ 
at all $\circ$.
Thus from Lemma \ref{le:chiral} (ii), we find 
\begin{eqnarray*}
{\mathcal E}_g(T_\infty(p))+
\sum_{i=1}^Lg(\xi^{(i)}\otimes p_i)
&=& \sum_{1\le j\le 3}I_g(u_\infty\otimes u_\infty\otimes 
p_1\otimes \cdots \otimes p_j)\\
&\overset{(\ref{uuI})}{=}& \sum_{1\le j\le 3}I_g(u_\infty\otimes 
p_1\otimes \cdots \otimes p_j)
\overset{(\ref{eig})}{=}
{\mathcal E}_g(p),
\end{eqnarray*}
completing the proof.
\end{proof}

{\it Proof of Theorem \ref{th:main}}.
{}From Remark \ref{re:u}, 
$T_\infty^t(p) = u_{l_1}\otimes \cdots \otimes u_{l_L}$
for $t\ge L$.
For such a state ${\mathcal E}_g=0$ and $\rho_g=0$ hold 
due to (\ref{zuu}) and (\ref{rr1})--(\ref{rr4}), respectively.
Thus it suffices to show 
\[
{\mathcal E}_g(p) - {\mathcal E}_g(T_\infty(p))
= \rho_g(p) - \rho_g(T_\infty(p)).
\]
By applying (\ref{edif}) and (\ref{rhog}) to the 
left and the right hand sides, respectively, this becomes
\[
\sum_{i=1}^Lg(\xi^{(i)}\otimes p_i)
= \sum_{i=1}^L(\gamma_g(p_i)
+\mathfrak{a}(p'_i)- \gamma_g(p'_i)),
\]
where we have set
$T_\infty(p) = p'_1\otimes \cdots \otimes p'_L$.
(This was denoted by
$p^1_1\otimes \cdots \otimes p^1_L$ in (\ref{rhog}).)
{}From the definition of $\xi^{(i)}$ in (\ref{edif}),
we have $\xi^{(i)}\otimes p_i \simeq p'_i\otimes \xi^{(i+1)}$.
Therefore Lemma \ref{le:ato} tells that 
$g(\xi^{(i)}\otimes p_i)
= \gamma_g(p_i)+\mathfrak{a}(p'_i)- \gamma_g(p'_i)$ 
holds for each $i$, finishing the proof.
\hfill$\square$ 

\subsection{$\ast$-transformed correspondence}\label{ss:cp}
Let us give an analogous result on $g=v^\ast_a\,(1 \le a \le n-2)$
which is not included in Theorem \ref{th:main}.
Our presentation in this subsection is brief since  
the essential features are the same as the previous case. 
To state the result, let us introduce a $\ast$-transformed
generalized energy and a $\ast$-transformed $D^{(1)}_n$
cellular automaton.
(See (\ref{sss}) and (\ref{sta}) for the original definition of $\ast$.)

Let $u^\ast_l=(0,\ldots, 0,l) \in B_l$. 
See (\ref{ul}).
The $\ast$-transformed generalized energy 
${\mathcal E}^\ast_{v^\ast_a}(p_L \otimes \cdots \otimes p_1)$ 
of an element $p_L \otimes \cdots \otimes p_1 \in 
B_{l_L}\otimes \cdots \otimes B_{l_1}$
is the sum of $v^\ast_a$ for all the vertices 
in the following diagram ($L=3$ example):
\begin{equation*}
\begin{picture}(80,65)(-50,1)
\put(-110,25)
{${\mathcal E}^\ast_{v^\ast_a}(p_3\otimes p_2 \otimes p_1) \;\;= $}

\put(50,55){$u^\ast_\infty$}
\put(29,55){$p_1$}
\put(13,55){$p_2$}
\put(-4,55){$p_3$}

\put(50,0){\line(-1,1){51}}
\put(16,51){\line(1,-1){32}}
\put(33,51){\line(1,-1){15}}
\qbezier(48,19)(50,17)(48,15)
\qbezier(48,32)(50,34)(48,36)
\put(33,0){\line(1,1){15}}
\put(16,0){\line(1,1){32}}
\put(-1,0){\line(1,1){51}}

\end{picture}
\end{equation*}
Compare this with Proposition \ref{pr:zushi}.
One can show that 
${\mathcal E}^\ast_{v^\ast_a}$ is well defined and $R$-invariant.

The $\ast$-transformed $D^{(1)}_n$ cellular automaton is the 
dynamical system on  $B_{l_L}\otimes \cdots \otimes B_{l_1}$
endowed with the commuting time evolutions $T^\ast_l (l\ge 1)$
defined by 
$p_L \otimes \cdots \otimes p_1\otimes u^\ast_l 
\simeq \xi \otimes T^\ast_l(p_L \otimes \cdots \otimes p_1)$.
($\xi \in B_l$ is determined by this relation.)
$T^\ast_\infty$ is well defined.
Moreover,  under the time evolution 
$p'_L \otimes \cdots \otimes p'_1
=T^\ast_\infty(p_L \otimes \cdots \otimes p_1)$,  
the equality $p'_j=u^\ast_{l_j}$ is valid for $1 \le j \le k+1$
if $p_j=u^\ast_{l_j}$ for $1 \le j \le k$, which 
is parallel with Remark \ref{re:u}. 
For $\zeta \in B_l$, introduce 
the charge conjugation of (\ref{all})--(\ref{gva}) by
\begin{eqnarray*}
{\mathfrak a}^\ast(\zeta) &:=& 
{\mathfrak a}(\zeta^\ast) = 
2\zeta_1+\zeta_2+\cdots + \zeta_n + \overline{\zeta}_n
+ \cdots + \overline{\zeta}_2,\\
\gamma^\ast_{v^\ast_a}(\zeta) &:=& \gamma_{v_a}(\zeta^\ast) = 
\zeta_1+\zeta_{a +1}+\cdots + \zeta_n + \overline{\zeta}_n
+ \cdots + \overline{\zeta}_2 \;\;(1 \le a \le n-2).
\end{eqnarray*}
Writing the time evolutions of 
a state $p=p_L \otimes \cdots \otimes p_1 \in 
B_{l_L}\otimes \cdots \otimes B_{l_1}$ as 
$(T^\ast_\infty)^t(p) =p^t_L \otimes \cdots \otimes p^t_1$,
we define the counting function ($1 \le a \le n-2$):
\begin{equation*}
\unitlength 0.1in
\begin{picture}( 12.9500, 12.000)(  -5.0500,-17.300)
\put(-17,-11){\begin{math}\displaystyle
\rho^\ast_{v^\ast_a}(p)
=\sum_{j=1}^L\gamma^\ast_{v^\ast_a}(p_j)
+ \sum_{t\ge 1}\sum_{j=1}^L{\mathfrak a}^\ast(p^t_j)\;\; =
\end{math}}

\special{pn 8}%
\special{pa 1942 600}%
\special{pa 982 600}%
\special{pa 982 760}%
\special{pa 1942 760}%
\special{pa 1942 600}%
\special{fp}%
\put(13.8100,-6.8000){\makebox(0,0){$\cdots$}}%
\put(18.3300,-6.8800){\makebox(0,0){$p_1$}}%
\put(16.3300,-6.8800){\makebox(0,0){$p_2$}}%
\put(11.20,-6.8800){\makebox(0,0){$p_L$}}%
%
\special{pn 8}%
\special{sh 0.300}%
\special{pa 1942 760}%
\special{pa 982 760}%
\special{pa 982 1720}%
\special{pa 1942 760}%
\special{ip}%
\end{picture}%
\end{equation*}
The counting is done by $\gamma^\ast_{v^\ast_a}$ for the top row and 
by ${\mathfrak a}^\ast$ 
for the SE quadrant generated by $T^\ast_\infty$
beneath it.
Thanks to the commutativity of $\ast$ and the combinatorial $R$
(Prop.4.4 in Ref. \cite{KOTY1}),  
Theorem \ref{th:main} implies the following:

\begin{proposition}\label{pr:main2}
For any state $p \in B_{l_L}\otimes \cdots \otimes B_{l_1}$,
the following equality is valid:
\begin{equation*}
{\mathcal E}^\ast_{v^\ast_a}(p) = \rho^\ast_{v^\ast_a}(p) \quad
(1 \le a \le n-2).
\end{equation*}
\end{proposition}

This completes our interpretation of the ($\ast$-transformed) 
generalized energies associated with all the generalized local energies 
in terms of the ($\ast$-transformed) $D^{(1)}_n$ cellular automaton. 

\section{Connection with combinatorial Bethe ansatz}\label{sec:re}

Combinatorial Bethe ansatz was initiated by  
Kerov-Kirillov-Reshetikhin (KKR)  \cite{KKR, KR}
to establish a fermionic formula of the Kostka-Foulkes polynomials
with the invention of rigged configurations and the KKR bijection.
Their fermionic formula generalized Bethe's formula  \cite{Be} 
for some simplest Kostka numbers that originates in the completeness issue.

Rigged configurations are combinatorial analogue of 
solutions to the Bethe equations.
The KKR bijection maps them to the combinatorial analogue of 
Bethe vectors which may be viewed as elements of (a subset of)
$B_{l_1}\otimes \cdots \otimes B_{l_L}$.
The combinatorial Bethe ansatz has flourished in  
the fermionic formulas for general affine Lie algebras  \cite{HKOTY,O,OSS,SS},
the solution of the initial value problem of  
integrable $A^{(1)}_n$ cellular automata 
by the inverse scattering method \cite{KOSTY}
and a connection with 
the classical soliton theory \cite{JM} 
via ultradiscrete tau functions \cite{KSY} and so forth.
Our aim in this section is to present an inverse scattering 
formalism of the $D^{(1)}_n$ cellular automaton 
and to conjecture explicit formulas for some 
generalized energies in the form of ultradiscrete tau functions 
associated with the $D^{(1)}_n$ rigged configurations.

\subsection{Inverse scattering formalism}\label{ss:ist}
Set 
${\mathcal P}_+ = \{p \in  B_{l_1} \otimes \cdots \otimes B_{l_L}
\mid {\tilde e}_ip=0 
\text{ for } i=1,2,\ldots, n\}$.
A state belonging to ${\mathcal P}_+$ is 
called {\em highest}.
It is known 
that there is a bijection 
between ${\mathcal P}_+$ 
and the set of {\em rigged configurations} \cite{OSS, SS}.
Consider a set 
\begin{equation}\label{S}
S = \{(a_i,j_i,r_i)\in \{1,2,\ldots, n\}\times 
\Z_{\ge 1}\times \Z_{\ge 0}
\mid i=1,2,\ldots, N\},
\end{equation}
where $N\ge 0$ is arbitrary and 
each triplet $s=(a,j,r)$ called {\em string} possesses 
color, length and rigging which will be 
denoted by ${\rm cl}(s)=a, {\rm lg}(s)=j$
and ${\rm rg}(s)=r$, respectively.\footnote{Colors 
$1,2,\ldots, n$ of strings in 
rigged configurations should not be confused with 
colors $1,2,\ldots, n, \overline{n},\ldots, \overline{2},
\overline{1}$ of (anti)-particles.}
$S$ is a rigged configuration if 
${\rm rg}(s) \le p^{({\rm cl}(s))}_{{\rm lg}(s)}$
is satisfied for all $s\in S$.
Here  
$p^{(a)}_j = \delta_{a,1}\sum_{k=1}^L\min(j,l_k)
-\sum_{t\in S}C_{a, {\rm cl}(t)}
\min(j,{\rm lg}(t))$, 
where 
$(C_{a,b})_{1\le a,b \le n}$ is the Cartan matrix of $D_n$.
Note that  
$ p^{({\rm cl}(s))}_{{\rm lg}(s)}\ge 0$ 
has to be satisfied for all $s \in S$, which 
imposes a stringent condition on the set 
$\{(a_i,j_i)\mid i=1,\ldots, N\}$.
Set ${\rm RC}=\{S: \text{rigged configuration}\}$.

\begin{theorem}[Refs. \cite{OSS} and \cite{SS}]\label{th:ss}
There is a bijection $\Phi:  {\mathcal P}_+ \rightarrow {\rm RC}$.
\end{theorem} 

An explicit algorithm to determine the 
image of $\Phi^{\pm 1}$ is known.
It is a $D^{(1)}_n$ analogue of 
the Kerov-Kirillov-Reshetikhin bijection \cite{KR} for 
$A^{(1)}_n$,
which plays a central role in the combinatorial Bethe ansatz.
Our convention here is the one adopted in Ref. \cite{KOSTY}.

For any highest state $p \in {\mathcal P}_+$,
its time evolution $T_l(p)$ is again highest.
Thus $T_l$ induces a time evolution on ${\rm RC}$ via $\Phi$.
Let $L$ be large enough and assume 
the situation in (\ref{eq:uka}).
Then we have  
\begin{theorem}\label{th:ist}
$\Phi(T_l(p)) = {\tilde T}_l \Phi(p)$ holds, where 
${\tilde T}_l: \{(a_i,j_i,r_i)\} \mapsto
\{(a_i,j_i,r_i+\delta_{a_i,1}\min(j_i,l))\}$
is a linear flow on rigged configurations.
\end{theorem}

\begin{proof}
The proof uses Theorem 8.6 of Ref. \cite{SS}
and is similar to that of Proposition 2.6 in 
Ref. \cite{KOSTY} for type $A^{(1)}_n$.
\end{proof} 

Thus the composition $\Phi^{-1}\circ {\tilde T}_l \circ \Phi$
linearizes the original time evolution $T_l$ and 
solves the initial value problem in the $D^{(1)}_n$ cellular automaton 
by the inverse scattering method.
See Ref. \cite{KOSTY} for an analogous result for $A^{(1)}_n$.

\subsection{Conjecture on ultradiscrete tau functions}\label{ss:tau}
For a rigged configuration $S$ (\ref{S}), 
let $T \subseteq S$ be a (possibly empty) subset of $S$.
In general, $T$ is no longer a rigged configuration.
We introduce the piecewise linear functions
($0 \le k \le L$ and $0 \le d \le n$)
\begin{eqnarray*}
c(T) &=& \frac{1}{2}\sum_{s,t \in T}C_{{\rm cl}(s),{\rm cl}(t)}
\min({\rm lg}(s), {\rm lg}(t)) 
+ \sum_{s\in T}{\rm rg}(s),\\
c^{(d)}_k(T) &=& c(T)
-\sum_{i=1}^k\sum_{s\in T, {\rm cl}(s)=1}
\min(l_i, {\rm lg}(s))
+\sum_{s\in T, {\rm cl}(s)=d}
{\rm lg}(s)
\end{eqnarray*}
By the definition, the last term in $c^{(d)}_k(T)$  is 
$0$ when $d=0$, and the relation 
\begin{equation}\label{yy}
c^{(d)}_k(T)=c^{(0)}_k(T)|_{{\rm rg}(s)\rightarrow 
{\rm rg}(s)+{\rm lg}(s)\delta_{{\rm cl}(s), d}}
\end{equation}
holds.
Obviously we have 
$c(\emptyset) = c^{(d)}_k(\emptyset) = 0$.
On the other hand,
$c(S)$ is known as the (co)charge of
the rigged configuration $S$ \cite{KR, HKOTY,OSS,SS}.

We define a 
$\Z_{\ge 0}$-valued piecewise linear function on $S$
as follows:
\begin{equation}\label{ttau}
\tau^{(d)}_k(S) = -\min_{T \subseteq S}
\left(c^{(d)}_k(T)\right)
\quad (0 \le k \le L,\;0 \le d \le n).
\end{equation}
For $S$ in (\ref{S}),
the minimum extends over $2^N$ 
candidates and reminds us of the structure of 
tau functions in the theory of solitons \cite{JM}.
In fact, for type $A^{(1)}_n$, analogous functions  
have been identified \cite{KSY} as ultradiscretization of the 
tau functions in KP hierarchy.
Although such an origin is yet to be clarified, 
we call (\ref{ttau}) {\em ultradiscrete tau function}.
Guided by the results in  
$A^{(1)}_n$
and supported by computer experiments, we propose
\begin{conjecture}\label{cj:tau}
For any highest state $p \in {\mathcal P}_+$,
let $S=\Phi(p)$ be the corresponding rigged configuration.
Then, the following equalities hold for 
$p_{[k]}$ (\ref{pk}) with $0 \le k \le L$.
\begin{eqnarray}
\tau^{(0)}_k(S) 
&=& {\mathcal E}_{v_0}(p_{[k]}),
\label{cj1}\\
\tau^{(1)}_k(S)
&=&{\mathcal E}_{v^{\sigma_1}_0}(p_{[k]}),
\label{cj2}\\
 \tau^{(n-1)}_k(S) 
&=&{\mathcal E}_{v^\ast_{n-1}}(p_{[k]}),
\label{cj3}\\
 \tau^{(n)}_k(S)
&=& {\mathcal E}_{v_{n-1}}(p_{[k]}),
\label{cj4}\\
\tau^{(2)}_k(S) 
&=&{\mathcal E}_{w_2}(p_{[k]}) -
{\mathcal E}_{v_0}(p_{[k]}) + \varphi_0(p_{[k]}).
\label{cj5}
\end{eqnarray}
\end{conjecture}
In (\ref{cj5}),  
$\varphi_0(p_{[k]})$ is the standard notation 
in crystal theory meaning 
$\max\{j\ge 0\mid \tilde{f}^j_0 p_{[k]}\neq 0\}$.
By using (\ref{nashi}), (\ref{rr1}) with $a=0$, 
(\ref{rr4}), (\ref{yy}) and Theorem \ref{th:ist}, 
one can show that (\ref{cj1}) and (\ref{cj2}) are equivalent.

The algorithm \cite{OSS,SS} 
for $\Phi^{\pm 1}$ seems valid not only for highest but  
{\em arbitrary} states if one allows negative rigging.
With such a generalization, we expect that Theorem \ref{th:ist} and 
Conjecture \ref{cj:tau} hold for any state,
which was indeed the case for type $A^{(1)}_n$ \cite{KSY}.

\begin{example}
For the initial state $p$ in Example \ref{ex:hatten2}, 
its rigged configuration is $S=\Phi(p) = 
\{(1, 8, -2), (1, 6, 0), (1, 2, -1), (1, 1, -1), (2, 8, 0)$, 
$(2, 6, -1),(2, 2, -1), (3, 8, -3), (4, 8, -1)\}$.
The ultradiscrete tau function 
$\tau^{(d)}_k(S)$ and $\varphi_0(p_{[k]})$ 
take the following values. 
\begin{equation*}
\begin{array}{|c|rcccccccc|}
\hline
 &k=1 \;& \, 2\; & \,3\; & \,4\; 
& 5 & 6 & 7 & 8 & 9\\
\hline
\tau^{(0)}_k(S)
& \;
6& 11& 21& 32& 39& 46& 53& 60& 67\\
\tau^{(1)}_k(S)
 &\;
0& 2& 7& 15& 22& 29& 36& 43& 50\\
\tau^{(2)}_k(S)
 &\;
1& 3& 9& 16& 23& 30& 37& 44& 51 \\
\tau^{(3)}_k(S)
 &\;
2& 6& 13& 24& 31& 38& 45& 52& 59\\
\tau^{(4)}_k(S)
 &\;
3& 8& 15& 24& 31& 38& 45& 52& 59\\
\varphi_0(p_{[k]})
 &\;
1& 0& 1& 0& 0& 0& 0& 0& 0\\
\hline
\end{array}
\end{equation*}
Comparing this with Example \ref{ex:soba}, one can check 
Conjecture \ref{cj:tau}.
\end{example}

Still many generalized energies in previous sections await 
formulas as in Conjecture \ref{cj:tau} to be discovered.
They are ultradiscrete analogue of the 
so called $X=M$ conjecture \cite{HKOTY,O}
in the sense that the generalized energies from crystal theory 
acquire explicit formulas of a fermionic nature 
originating in the combinatorial Bethe ansatz.
Such results combined with Theorems \ref{pr:rho}
and \ref{th:main} will lead to a piecewise linear formula
for the bijection $\Phi^{-1}$  as was done for $A^{(1)}_n$ \cite{KSY}.

Let us end by raising a closely related question as another future problem.
Let ${\mathcal P}_+(\lambda)$ denote the 
subset of ${\mathcal P}_+$ having the prescribed weight $\lambda$.
Then, does the generating function 
\begin{equation*}
X_g(\lambda) 
= \sum_{p \in {\mathcal P}_+(\lambda)}q^{{\mathcal E}_g(p)}
\end{equation*}
of the generalized energy 
admit a fermionic formula like Refs. \cite{KR} and \cite{HKOTY}?

\section*{Acknowledgments}
The authors thank Masato Okado for discussion.
This work is supported by Grants-in-Aid for 
Scientific Research No. 19540393, 
No. 21740114 and No. 17340047 from JSPS.


\begin{thebibliography}{9}

\bibitem{ABF}
G.~E.~Andrews, R.~J.~Baxter and P.~J.~Forrester,
{\it J. Stat. Phys. } {\bf 35} 193 (1984).

\bibitem{Ba}
R. J. Baxter, 
{\it Exactly solved models in statistical mechanics}, 
(Dover, 2007).

\bibitem{BK}
A.~Berenstein and D.~Kazhdan,
\textit{Geometric and unipotent crystals},
GAFA 2000 (Tel Aviv, 1999), 
Geom. Funct. Anal. 2000, Special Volume, Part I, pp188--236.

\bibitem{Be}
H.~A.~Bethe,
{\it Z. Physik} {\bf 71} 205 (1931).

\bibitem{HKOTY}G.~Hatayama, A.~Kuniba, M.~Okado, 
T.~Takagi and Y.~Yamada, 
{\it Remarks on fermionic formula}, 
in  {\it Recent Developments in Quantum Affine Algebras 
and Related Topics},  eds.  N.~Jing and K.C.~Misra, 
Contemp. Math {\bf 248} (AMS 1999) pp243-291. 

\bibitem{HKT1}G.~Hatayama, A.~Kuniba and T.~Takagi,
{\it Nucl. Phys. B}\ {\bf 577} [PM] 619 (2000).

\bibitem{HKT2}G.~Hatayama, A.~Kuniba and T.~Takagi,
{\it J. Phys. A}\  {\bf 34} 10697 (2001).  

\bibitem{JM}
M.~Jimbo and T.~Miwa,
{\it Publ. RIMS}. {\bf 19} (Kyoto Univ. 1983) pp.943--1001.

\bibitem{KKM}
S-J.~Kang, M.~Kashiwara and K.~C.~Misra,
{\it Compositio Math.} {\bf 92} 299 (1994).

\bibitem{KMN1}
S-J.~Kang, M.~Kashiwara, K.~C.~Misra,
T.~Miwa, T.~Nakashima and A.~Nakayashiki,
\textit{Int.\ J.\ Mod.\ Phys. A}
{\bf 7} (suppl. 1A) 449 (1992).

\bibitem{KMN2}
S-J.~Kang, M.~Kashiwara, K.~C.~Misra, 
T.~Miwa, T.~Nakashima and A.~Nakayashiki,
{\it Duke Math.\ J.}\ {\bf 68} 499 (1992).

\bibitem{KKR}
S.~V.~Kerov, A.~N.~Kirillov and N.~Yu.~Reshetikhin, 
{\it J. Soviet Math.} {\bf 41} 916 (1988). 

\bibitem{KR}
A.~N.~Kirillov and N.~Yu.~Reshetikhin, 
{\it J. Soviet Math.} {\bf 41} 925 (1988).

\bibitem{KOSTY}
A.~Kuniba, M.~Okado, R.~Sakamoto, T.~Takagi and Y.~Yamada:
{\it Nucl. Phys. B}\ {\bf 740} [PM] 299 (2006).

\bibitem{KOTY1}
A.~Kuniba, M.~Okado,  T.~Takagi and Y.~Yamada,
\textit{Int. Math. Res. Notices.} {\bf 48} 2565 (2003).

\bibitem{KOTY2}
A.~Kuniba, M.~Okado,  T.~Takagi and Y.~Yamada,
\textit{Commun. Math. Phys.} {\bf 245} 491 (2004).

\bibitem{KSY}
A.~Kuniba, R.~Sakamoto and Y.~Yamada,
\textit{Nucl. Phys. B} {\bf 786} [PM] 207 (2007).

\bibitem{NY}
A.~Nakayashiki and Y.~Yamada,
{\it Selecta Math.}, New Ser. {\bf 3} 547 (1997).

\bibitem{O}
M.~Okado, {\it $X=M$ conjecture}, in {\it Combinatorial 
Aspect of Integrable Systems}, eds. A.~Kuniba and M.~Okado,
MSJ Memoir Vol. 17 (MSJ 2007) pp43--73.

\bibitem{OSS}
M.~Okado, A.~Schilling and M.~Shimozono,
{\it A crystal to rigged configuration bijection 
for nonexceptional affine algebras}, in 
{\it Algebraic Combinatorics and Quantum Groups},
eds. N.~Jing, (World Scientific 2003), pp85--124. 

\bibitem{SS}
A.~Schilling and M.~Shimozono,
{\it J. Alg. }{\bf 295} 562 (2006).

\bibitem{Sh}
M.~Shimozono, private communication.


\end{thebibliography}
\end{document}